\documentclass[11pt,oneside]{amsart}
\usepackage{amsmath, amssymb, amsthm}

\newtheorem{theorem}{Theorem}
\newtheorem{lemma}[theorem]{Lemma}

\newtheorem{corollary}[theorem]{Corollary}
\newtheorem*{question}{Question}

\title{Every Finite Group Admits a Just Finite Presentation}
\author{Marc Lackenby}
\address{Mathematical Institute, University of Oxford, \newline Woodstock Road, Oxford OX2 6GG, United Kingdom}
\thanks{The author was partially supported by EPSRC grant EP/Y004256/1. For the purpose of open access, the author has applied a CC BY public copyright licence to any author accepted manuscript arising from this submission. See the Methodology section for details of the use of AI in the creation of this paper.}

\begin{document}

\maketitle

\begin{abstract}
A finite presentation $\langle X \, | \, R \rangle$ of a finite group $G$ is called ``just finite'' if removing any relation from $R$ results in a presentation for an infinite group. It has been an open question (Kourovka Notebook, Problem 21.10) whether every finite group admits such a presentation. We resolve this conjecture in the affirmative. Significant assistance was provided by the Google DeepMind AI co-mathematician tool. %The proof proceeds by splitting into two cases: cyclic groups, which trivially admit balanced just finite presentations, and non-cyclic groups, for which we provide an explicit construction utilizing B.H. Neumann's presentation of the trivial group. 
\end{abstract}

\section{Introduction}
Inspired by the extensive theory of just infinite groups \cite{McCarthy1, McCarthy2, Wilson, Wilson2000, Grigorchuk2000, Reid:Structure, Reid2010}, Barnea introduced a notion that is, in some sense, dual \cite{Barnea}. He defined a \emph{just finite} presentation of a finite group $G$, which is a presentation $\langle X \, | \, R \rangle$ for the group with the property that, for every $r \in R$, the group $\langle X \, | \, R \setminus \{r\} \rangle$ is infinite. In  Problem 21.10 of the Kourovka Notebook \cite{Kourovka}, Barnea asked whether every finite group admits a just finite presentation. The goal of this note is to prove that this is indeed the case.

\begin{theorem}
\label{Thm:Main}
Every finite group admits a finite, just finite presentation.
\end{theorem}

This is actually a consequence of a stronger result.
Recall that a group $G$ has \emph{Property (FA)} if any action of $G$ on a tree has a fixed point. It is a fundamental theorem of Serre \cite{Serre} that a finitely generated group has
Property (FA) if and only if it is neither a non-trivial amalgamated free product nor an HNN extension. We say that a presentation $\langle X \, | \, R \rangle$ for a group with Property (FA) is \emph{just-(FA)} if, for every $r \in R$, the group $\langle X \, | \, R \setminus \{r\} \rangle$ does not have Property (FA).

\begin{theorem}
\label{Thm:JustFA}
Every finitely presented group with Property (FA) admits a finite just-(FA) presentation.
\end{theorem}

This implies Theorem \ref{Thm:Main}, since a finite group $G$ clearly has Property (FA). Hence, it admits a finite just-(FA) presentation, which is therefore just finite.

More generally, given any property $P$ of a group, we say that a finite presentation $\langle X \, | \, R \rangle$ of a group with property $P$ is \emph{just-$P$} if, for every $r \in R$, the group $\langle X \, | \, R \setminus \{r\} \rangle$ does not have property $P$. Among the properties of a group $G$ that we consider are Kazhdan's Property (T) \cite{Kazhdan}. Watatani proved that Property (T) implies Property (FA) for countable discrete groups \cite{Watatani}. Hence, Theorem \ref{Thm:JustFA} has the following corollary.

\begin{corollary}
\label{Cor:JustT}
If a finitely presented group has Property (T), then it has a finite just-(T) presentation.
\end{corollary}

We say that a finitely generated group is \emph{small} if it does not contain a non-abelian free subgroup. The referee observed that the proof of Theorem \ref{Thm:JustFA} can be adapted to give the following result.

\begin{theorem}
\label{Thm:JustSmall}
Every finitely presented small group admits a finite just-small presentation.
\end{theorem}

This has applications to various classes of groups. Every finitely generated amenable group is small. Note that a finitely generated virtually polycyclic group is amenable. Similarly, finitely generated solvable groups are amenable. Hence, these groups are also small. Thus, Theorem \ref{Thm:JustSmall} has the following corollary.

\begin{corollary}
\label{Cor}
Let $G$ be a finitely presented group.
\begin{enumerate}
\item If $G$ is amenable, then it has a finite just-amenable presentation.
\item If $G$ is solvable, then it has a finite just-solvable presentation.
\item If $G$ is virtually polycyclic, then it has a finite just-(virtually polycyclic) presentation.
\end{enumerate}
\end{corollary}

The above corollary was observed by the referee. The author is grateful for their significant contribution to this paper.

\section{The Construction}

The aim is to start with a finite presentation $\langle X \, | \, R \rangle$ of the given group $G$ with Property (FA) or that is small, and modify it to a just-(FA) or just-small presentation. Each relation will be substituted for two new relations that have a form based on the following lemma.

\begin{lemma}[B.H. Neumann \cite{Neumann}]
\label{lem:neumann}
Let $H$ be a group with elements $u, v \in H$ satisfying $u^{-1}vu = v^2$ and $v^{-1}uv = u^2$. Then $u=1$ and $v=1$.
\end{lemma}

\begin{proof}
From the second relation, $v^{-1}uv = u^2$, we can write the commutator $[u, v] = u^{-1}(v^{-1}uv) =
u^{-1}u^2 = u$. From the first relation, $u^{-1}vu = v^2$, we can invert it to get $u^{-1}v^{-1}u =
v^{-2}$. We can then write the commutator as $[u, v] = (u^{-1}v^{-1}u)v = v^{-2}v = v^{-1}$. Equating
the two expressions for the commutator, we have $u = v^{-1}$. Substituting this into the first relation
$u^{-1}vu = v^2$ yields $(v^{-1})^{-1}v(v^{-1}) = v^2$, which simplifies to $v = v^2$. This implies $v = 1$, and therefore
$u = 1$.
\end{proof}

The following type of semi-direct product of cyclic groups plays a technical role in our proof.

\begin{lemma}
\label{Lem:SemiDirect}
For $k > 1$, the presentation $\langle x, y \, | \, x^{-1} y x = y^2, x^k = 1 \rangle$ specifies the group $\mathbb{Z}_{2^k-1} \rtimes \mathbb{Z}_k$, where $\langle x \rangle = \mathbb{Z}_k$ and $\langle y \rangle = \mathbb{Z}_{2^k-1}$ and the conjugation action of $x$ on $\langle y \rangle$ is $y \mapsto y^2$.
\end{lemma}

\begin{proof}
There is a homomorphism $\langle x, y \, | \, x^{-1} y x = y^2, x^k = 1 \rangle \rightarrow \mathbb{Z}_{2^k-1} \rtimes \mathbb{Z}_k$, since the two relations are satisfied in the image group. It is surjective because the images of $x$ and $y$ generate $\mathbb{Z}_{2^k-1} \rtimes \mathbb{Z}_k$. Any element in $\langle x, y \, | \, x^{-1} y x = y^2, x^k = 1 \rangle$ can be written as $x^m y^n$ for some integers $m$ and $n$. Using $x^k = 1$, we may assume $0 \leq m < k$. Conjugating $y$ by $x$ a total of $k$ times gives $y^{2^k}$. Since $x^k = 1$, we deduce that $y^{2^k} = y$. So, we may also assume $0 \leq n < 2^{k}-1$. So $\langle x, y \, | \, x^{-1} y x = y^2, x^k = 1 \rangle$ has at most $k(2^k -1)$ elements. However, $\mathbb{Z}_{2^k-1} \rtimes \mathbb{Z}_k$ has exactly this many elements, and therefore the above surjective homomorphism is an isomorphism.
\end{proof}

We are now in a position to prove our main theorems.

\begin{proof}[Proof of Theorems \ref{Thm:JustFA} and \ref{Thm:JustSmall}]
Let $G$ be a finitely presented group with Property (FA) or that is small.

Suppose first that $G$ is the cyclic group $\mathbb{Z}_n$ for $n \geq 0$. It then admits the presentation 
$$\langle X'' \, | \, R'' \rangle = \langle x,y \, | \, x^{n+3} y^{-3}, x^{n+4} y^{-4} \rangle.$$
This is indeed a presentation of $\mathbb{Z}_n$, since we may modify the presentation using a sequence of Tietze transformations as follows:
\begin{align*}
& \langle x,y \, | \, x^{n+3} y^{-3}, x^{n+4} y^{-4} \rangle \\
\rightarrow \quad
& \langle x,y \, | \, x^{n+4} y^{-3}x^{-1}, x^{n+4} y^{-4} \rangle \\
\rightarrow \quad
& \langle x,y \, | \, yx^{-1}, x^{n+4} y^{-4} \rangle \\
\rightarrow \quad
& \langle x \, | \, x^{n} \rangle.
\end{align*}
Furthermore, when $n \not=0$, this presentation $\langle X'' \, | \, R'' \rangle$ is just-(FA), since removing one of the relations gives a group $G'$ that surjects onto
$$\langle x,y \, | \, x^{n+k} y^{-k}, x^{n+k}, y^{k} \rangle$$
where $k \in \{ 3,4 \}$. This latter group is a free product $\mathbb{Z}_{n+k} \ast \mathbb{Z}_k$, and so $G'$ acts on a tree without a fixed point. Furthermore, because
$n+k \geq k > 2$, $G'$ contains a non-abelian free subgroup, and so $\langle X'' \, | \, R'' \rangle$ is just-small.

So assume $G$ is non-cyclic. Let $\langle X \, | \, R \rangle$ be an irredundant finite presentation for $G$ (that is, for any $r \in R$, the group $H_r = \langle X \, | \, R \setminus \{r\} \rangle$ is not isomorphic to $G$, implying $r \neq 1$ in $H_r$). 

We define a new presentation $P' = \langle X' \, | \, R' \rangle$ as follows:
\begin{itemize}
    \item $X' = X \cup \{b_r \, | \, r \in R\}$
    \item $R' = \{r^{-1} b_r r = b_r^2 \, | \, r \in R\} \cup \{b_r^{-1} r b_r = r^2 \, | \, r \in R\}$
\end{itemize}
By Lemma \ref{lem:neumann}, these relations force $b_r = 1$ and $r = 1$ for all $r \in R$. (See Section \ref{Sec:Examples} for a specific example of such a presentation $P'$.)

We claim that $P'$ presents $G$. Consider what happens to the group $G = \langle X \, | \, R \rangle$ when a new generator $b_r$ is added, the relation $r$ is removed and then replaced by 
the two new relations $r^{-1} b_r r = b_r^2$ and $b_r^{-1} r b_r = r^2$. These force $b_r = 1$ and $r = 1$. Conversely $b_r = 1$ and $r = 1$ imply the two new relations 
$r^{-1} b_r r = b_r^2$ and $b_r^{-1} r b_r = r^2$. Thus, instead of adding these two new relations, we could have simply set $b_r = 1$ and $r = 1$. In other words, the new group is obtained from $G = \langle X \, | \, R \rangle$ by adding the new generator $b_r$ and a new relation $b_r = 1$. Since this just modifies $\langle X \, | \, R \rangle$ by a single Tietze transformation, 
this does not change the group. Repeating this for each relation in $R$, we deduce that $P'$ does present $G$.

%We must show $P'$ is just-(FA). 
Let $r \in R$. We show that removing either of the corresponding relations of $R'$ yields a group without Property (FA) and containing a non-abelian free subgroup.

\textbf{Case 1: Remove $b_{r}^{-1} r b_{r} = r^2$.}
Let $K_1$ be the resulting group. As above, let $H_r$ be the group $\langle X \, | \, R \setminus \{ r \} \rangle$. This also has a presentation obtained from $P'$ by removing the generator $b_r$ and the relations  $r^{-1} b_{r} r = b_{r}^2$ and $b_r^{-1} r b_r = r^2$. Let $k \in \mathbb{N} \cup \{ \infty \}$ be the order of $r$ in $H_r$. Since $R$ is irredundant, $k > 1$. Thus $K_1$ can be decomposed as an amalgamated free product:
\[ K_1 \cong H_{r} *_{\langle r \rangle = \langle x \rangle} B' \]
where $B' = \langle x, b_{r} \, | \, x^{-1} b_{r} x = b_{r}^2, x^k = 1 \rangle$. When $k = \infty$, we do not include the relation $x^k = 1$. Thus, when $k = \infty$, $B'$ is the Baumslag-Solitar group $BS(1,2)$. When $k$ is finite,
Lemma \ref{Lem:SemiDirect} gives that the group $B'$ is the semi-direct product $\mathbb{Z}_{2^k-1} \rtimes \mathbb{Z}_k$, and the index of the amalgamating subgroup $\langle x \rangle$ in $B'$ is $2^k - 1$. Thus, when $k = \infty$ or $k \ge 2$, this index is at least $3$.

An amalgamated free product $A *_C B$ is non-trivial when $C \not= A$ and $C \not= B$. Moreover, when the amalgamated free product is non-trivial and the index of $C$ in $A$ or $B$ is at least 3, then $A *_C B$ contains a free non-abelian group.
Since $[B' : \langle x \rangle] \ge 3 > 1$, $K_1$ is a non-trivial amalgamated free product unless $H_{r} = \langle r \rangle$.
In this case, $H_{r}$ is a cyclic group. However, $G$ is the quotient of $H_{r}$ by the normal closure of $r$. Since any quotient of a cyclic group is cyclic, this contradicts the assumption that $G$ is non-cyclic. Therefore, $H_{r} \neq \langle r \rangle$, and $K_1$ is a non-trivial amalgamated free product containing a free non-abelian subgroup.

\textbf{Case 2: Remove $r^{-1} b_{r} r = b_{r}^2$.}
Let $K_2$ be the resulting group. The remaining relation is $b_{r}^{-1} r b_{r} = r^2$.
We define a homomorphism $\phi: K_2 \to G \ast \mathbb{Z} = G \ast \langle t \rangle$ by sending each $x \in X$ and $b_{r'}$ ($r' \not =r$)
to the corresponding element of $G$, and sending $b_r$ to $t$. All the defining relations for $K_2$ other than $b_{r}^{-1} r b_{r} r^{-2}$
are trivially sent to the identity element of $G$. The remaining relation maps to $t^{-1} \cdot 1 \cdot t = 1$. Hence, this homomorphism is well-defined and surjective. So, $K_2$ does not have Property (FA). Moreover, because $G$ is non-trivial, $G \ast \mathbb{Z}$ contains a non-abelian free subgroup, and hence this is also true of $K_2$.

So $P'$ is the required just-(FA) or just-small presentation for $G$.
\end{proof}

It is worth observing that, when $G$ is small and has Property (FA), a single presentation is built for $G$ that is both just-small and just-(FA). Furthermore, the hypothesis that $G$ has Property (FA) and is small is not really used. So we have actually proved the following.

\begin{theorem}
Every finitely presented group has a finite presentation $\langle X \, | \, R \rangle$ with the following property. For each $r \in R$, the group $\langle X \, | \, R \setminus \{ r \}\rangle$ contains a non-abelian free subgroup and does not have Property (FA).
\end{theorem}

We remark, as a consequence of the above proof, the following result. Recall that the \emph{deficiency} of a presentation $\langle X \, | \, R \rangle$ is $|X| - |R|$.

\begin{theorem}
Let $P$ be any finite irredundant presentation of a non-cyclic group $G$ with Property (FA). Then $G$ admits a finite, just-(FA) presentation $P'$ where the deficiency of $P'$ is equal to that of $P$, and where the generators of $P'$ represent the same elements of $G$ that the generators of $P$ did, plus multiple copies of the identity element.
\end{theorem}

For simplicity of exposition, we focused on just-(FA) presentations in the above statement, but the analogous result also holds for the just-small presentations that we construct.
Although the deficiency of the presentation remains unchanged, the number of generators increases. Indeed, if $\langle X \, | \, R \rangle$ is the initial presentation $P$, then the presentation $P'$ that is constructed has $|X| + |R|$ generators. It would be interesting to investigate whether just-(FA) presentations with fewer than $|X| + |R|$ generators can be found. For example, the following question is of interest.

\begin{question}
Does every finite non-cyclic group $G$ admit a just finite presentation $\langle X \, | \, R \rangle$ where $|X|$ is the minimal cardinality of a generating set for $G$?
\end{question}

\section{Examples}
\label{Sec:Examples}
The following is a presentation of the dihedral group of order 8:
$$\langle \sigma, \tau \, | \, \sigma^4 = 1, \tau^2 = 1, \tau^{-1} \sigma \tau = \sigma^3 \rangle.$$
When the first relation is removed, this gives a semi-direct product $\mathbb{Z}_{8} \rtimes \mathbb{Z}_2$, by an argument analogous to the one in Lemma \ref{Lem:SemiDirect}. When the second relation is removed, the resulting group admits a surjective homomorphism onto $\mathbb{Z} = \langle t \rangle$, sending $\sigma$ to $1$ and $\tau$ to $t$. When the third relation is removed, we have the free product $\mathbb{Z}_4 \ast \mathbb{Z}_2$. So this presentation is irredundant but not just finite. When we apply the above procedure to it, we obtain the just finite presentation
$$\left \langle \sigma, \tau, a, b, c \;\middle|\;
\begin{array}{c} \sigma^{-4} a \sigma^4 = a^2 , a^{-1} \sigma^4 a = \sigma^8, \\
\tau^{-2} b \tau^2 = b^2 , b^{-1} \tau^2 b = \tau^4, \\
(\tau^{-1} \sigma \tau \sigma^{-3})^{-1} c (\tau^{-1} \sigma \tau \sigma^{-3}) = c^2,\\
c^{-1} (\tau^{-1} \sigma \tau \sigma^{-3}) c= (\tau^{-1} \sigma \tau \sigma^{-3})^2 
\end{array}
\right \rangle.$$
However, the following presentation for the dihedral group of order 8 is already just finite:
$$\langle \sigma, \tau \, | \, \sigma^4 = 1, \tau^2 = 1, \tau^{-1} \sigma \tau = \sigma^{-1} \rangle.$$
Removing the first relation gives the infinite dihedral group.

\section{Methodology}
The first version of this paper was produced using the AI co-mathematician tool \cite{Zheng} developed by Google DeepMind. This system, which is not at present publicly available, is based on Gemini 3.1 Pro. The author was given access under Google's trusted tester programme.

 The author posed Problem 21.10 from the Kourovka Notebook to the AI co-mathematician. AI co-mathematician proposed a solution, which was essentially the proof given in this paper, but then found a flaw in its argument which it was unable to resolve. The issue was that it was uncertain how to proceed in the case where $H_r = \langle r \rangle$  in Case 1 of the proof of Theorem \ref{Thm:JustFA}. By analysing the argument carefully, the author was able to find a fix for it, which the AI co-mathematician then confirmed was correct and wrote up. The author then made some minor edits and clarifications. The author then introduced the notions of just-(FA) and just-(T) presentations, and made the necessary minor adjustments to the proofs. The paper was then checked and corrected by the AI co-mathematician tool, and then the author submitted it.

The referee kindly observed that the proof could be extended a little to give Theorem \ref{Thm:JustSmall} and suggested Corollary \ref{Cor} as a consequence.
To incorporate these additions, the author then re-wrote the introduction of the paper and made adjustments to the proof. 
The author is indebted to the referee for their generous input.

\bibliography{just-finite-biblio}
\bibliographystyle{plain}

\end{document}